\documentclass[a4paper,11pt]{amsart}
\usepackage{amsfonts}
\usepackage{amsthm}
\usepackage{amsmath}
\usepackage{amssymb}
\usepackage{graphicx}
\usepackage{enumitem}
\usepackage{color}
\usepackage{mathrsfs} 
\usepackage{stmaryrd}

\definecolor{sealbrown}{rgb}{0.2, 0.08, 0.08}
\usepackage[colorlinks=true]{hyperref}
\hypersetup{citecolor=blue, linkcolor=blue}

\theoremstyle{plain}
\newtheorem{problem}{Problem}
\newtheorem*{prob*}{Problem}
\theoremstyle{definition}
\newtheorem{defin}{Definition}[section]
\newtheorem{theorem}{Theorem}[section]
\newtheorem{proposition}{Proposition}[section]
\newtheorem{cor}{Corollary}[section]
\newtheorem{lemma}{Lemma}[section]
\theoremstyle{definition}
\newtheorem{example}{Example}[section]

\theoremstyle{remark}
\newtheorem{remark}[theorem]{Remark}

\newcommand{\adef}{\begin{defn}}
\newcommand{\zdef}{\end{defn}}
\newcommand{\diam}{\mathrm{diam}}
\newtheorem{defn}[theorem]{Definition}

\def\co{\operatorname{c}_0}

\newcommand{\aproof}{\begin{proof}}
\newcommand{\zproof}{\end{proof}}

\makeatletter
\def\@tocline#1#2#3#4#5#6#7{\relax
    \ifnum #1>\c@tocdepth 
    \else
    \par \addpenalty\@secpenalty\addvspace{#2}%
    \begingroup \hyphenpenalty\@M
    \@ifempty{#4}{%
        \@tempdima\csname r@tocindent\number#1\endcsname\relax
    }{%
        \@tempdima#4\relax
    }%
    \parindent\z@ \leftskip#3\relax \advance\leftskip\@tempdima\relax
    \rightskip\@pnumwidth plus4em \parfillskip-\@pnumwidth
    #5\leavevmode\hskip-\@tempdima
    \ifcase #1
    \or\or \hskip 1em \or \hskip 2em \else \hskip 3em \fi%
    #6\nobreak\relax
    \hfill\hbox to\@pnumwidth{\@tocpagenum{#7}}\par
    \nobreak
    \endgroup
    \fi}
\makeatother

\title[Asymptotically H\"older nonexpansive mappings]{Fixed point results for mappings of asymptotically H\"older nonexpansive type}

\author{C. S. Barroso}
\address{Departamento de Matem\'atica, Universidade Federal do Cear\'a, Campus do Pici, Bl 914
Fortaleza, CE 60455-900, Av Humberto Monte S/N Brazil}
\email{cleonbar@mat.ufc.br}

\author{C. S. R. da Silva}
\address{Departamento de Matem\'atica, Universidade Federal do Cear\'a, Campus do Pici, Bl 914
Fortaleza, CE 60455-900, Av Humberto Monte S/N Brazil}
\email{csergiomat@gmail.com}



\thanks{This research has partially been supported by CAPES}

\begin{document}

\maketitle

\begin{abstract}
We introduce asymptotically H\"older nonexpansive mappings and mappings of asymptotically H\"older nonexpansive type, in which the H\"older exponents converge to one along the iterates. We prove that if \(K\) is a nonempty closed bounded convex subset of a Banach space \(X\) with characteristic of convexity \(\epsilon_0(X)<1\), then every self-mapping of \(K\) of asymptotically H\"older nonexpansive type admits a point \(x\in K\) such that \(T^n x\to x\). Consequently, a fixed point exists whenever some positive iterate of \(T\) is continuous; in particular, this applies to asymptotically H\"older nonexpansive mappings. We also obtain fixed-point-free constructions in spaces containing copies of \(c_0\), and a related construction on
a bounded convex, not necessarily closed, subset of every Banach space containing an isomorphic copy of \(\ell_1\). Further examples show that the new classes properly extend their classical counterparts and may contain mappings with no continuous positive iterate. Finally, we examine Lin's renorming of \(\ell_1\), identify an obstruction related to shift-type constructions, and establish a shrinking-diagonal criterion that reduces the remaining closed-set problem to the
construction of a suitably controlled uniformly Lipschitzian mapping.
\end{abstract}

\section{Introduction}

Fixed point theorems for nonlinear operators are fundamental tools across mathematics and engineering. A central concept is the fixed point property (FPP), which guarantees that every mapping within a given class on a nonempty, closed, bounded, convex set $K$ has a fixed point. For nonexpansive mappings, this line of research originated in 1965: Browder established the FPP for Hilbert spaces, Browder and Göhde independently proved it for uniformly convex Banach spaces, and Kirk extended the result to reflexive spaces with normal structure. Expanding beyond Banach's contraction principle, these foundational results sparked extensive research into how geometric properties influence the FPP (see \cite{GK2, KSims}).

\smallskip 

In \cite{GK1}, Goebel and Kirk introduced the class of asymptotically nonexpansive mappings. A mapping $T$ is said to be asymptotically nonexpansive if the Lipschitz constants of its iterates converge to $1$ as $n \to \infty$. That is,
\[
\|T^n(x) - T^n(y)\| \le (1+\kappa_n)\|x - y\|, \qquad \forall\, x, y \in K,\,n\in\mathbb{N},
\]
where $(\kappa_n)\subset [0,\infty)$ and $\kappa_n \to 0$. 

They proved the FPP for these mappings in the setting of  uniformly convex spaces. Since then, the study of the FPP for such mappings has developed in several directions. As highlighted in \cite{Ki-MY-S}, the interest in this class of mappings lies in their ability to test the limits of nonexpansive theory. Following the seminal works of Goebel and Kirk, later studies focused on loosening both underlying geometric conditions and mapping regularity. For example, Kirk \cite{Ki1} relaxed the requirement of uniform convexity, extending the Goebel-Kirk theorem to spaces satisfying the weaker condition $\epsilon_0(X) < 1$ (Definition \ref{def:2.2}). This result opened the door to exploring the FPP within broader classes of reflexive Banach spaces. In another direction, developments shifted toward \textit{mappings of asymptotically nonexpansive type}, where the restriction of a uniform Lipschitz constant for each iterate is replaced by a more flexible asymptotic condition:
\begin{equation}\label{eq:intro_type}
\limsup_{n \to \infty} \Big\{ \sup_{y \in K} \big\{ \|T^n(x) - T^n(y)\| - \|x - y\| \big\} \Big\} \le 0, \quad \forall\, x \in K.
\end{equation}

\noindent Current research develops these questions in several directions. These include extensions to nonlinear settings, such as $\mathrm{CAT}(0)$ and hyperbolic spaces (cf. \cite{Ki2,Ki3,Ki4, KPa}), as well as the study of mappings under relaxed notions of nonexpansivity \cite{Bar1, Bar2, DBen1, Ki1, Kirk, Tingley}. For instance, \cite[Theorem 4.17(ii)]{Bar2} shows that for any $\alpha\in (0,1)$ the unit ball $B_X$ of any Banach space $X$ admits a fixed point-free mapping $T\colon B_X\to B_X$ satisfying $\inf_{y\in B_X}\|y - T(y)\|=0$ and 
\[
\limsup_{n\to\infty}\Big\{ \sup_{y\in B_X}\|T^n(x) - T^n(y)\| - \|x - y\|^\alpha\Big\}\leq 0,\quad\forall x\in B_X.
\]
Such developments highlight a growing interest in the interplay between the FPP and variant notions of asymptotic nonexpansivity (cf. \cite{DBen1, Ki-MY-S} and references therein).

\subsection{Goal.} The primary objective of this paper is to extend these classical results to a broader class of mappings, whose regularity is characterized by a sequence of H\"older exponents $\alpha_n$ converging to $1$.

\begin{defin}\label{def:1.1}
Let \(K\) be a nonempty subset of \(X\), and let \(T\colon K\to K\) be a mapping.

\begin{enumerate}[label=\textnormal{(\roman*)}]
\item We say that \(T\) is \emph{asymptotically H\"older nonexpansive} if there exist sequences \((\alpha_n)\subset(0,1]\) and \((\varepsilon_n)\subset[0,\infty)\) such that $\alpha_n\to1$, $\varepsilon_n\to0$, and
\[
\|T^n x-T^n y\|\leq(1+\varepsilon_n)\|x-y\|^{\alpha_n}
\]
for all \(x,y\in K\) and \(n\geq 1\).\\[1.5mm]

\item Given \(C\geq1\), we say that \(T\) is \emph{uniformly asymptotically H\"older-Lipschitz with constant \(C\)} if
there exists a sequence \((\alpha_n)\subset(0,1]\), with \(\alpha_n\to1\), such that
\[
\|T^n x-T^n y\|\leq C\|x-y\|^{\alpha_n}
\]
for all \(x,y\in K\) and \(n\geq1\).\\[.5mm]

\item We say that \(T\) is \emph{of asymptotically H\"older nonexpansive type} if there exists a sequence
\((\alpha_n)\subset(0,1]\), with \(\alpha_n\to1\), such that
\begin{equation}\label{eqn:1.2}
\limsup_{n\to\infty}\sup_{y\in K}\left(\|T^n x-T^n y\|-\|x-y\|^{\alpha_n}\right)\leq0
\end{equation}
for every \(x\in K\).
\end{enumerate}
\end{defin}

\medskip 

When the exponent sequence must be emphasized, we use the prefix \((\alpha_n)\)-H\"older.

\begin{remark}
Note that in the case when $(\alpha_n) \equiv 1$, the conditions above reduce to their respective classical counterparts. Also, Tingley's example \cite{Tingley} shows that the supremum
condition in \eqref{eqn:1.2} cannot simply be suppressed. 
\end{remark}

\medskip

\subsection{Organization of the paper}  Section \ref{sec:2} provides the necessary preliminaries. In Section \ref{sec:3} we establish our main results. In Section \ref{sec:4}, we study the FPP for the class of maps considered in this paper under Lin's renorming for $\ell_1$. Section \ref{sec:5} exhibits several examples of mappings within our proposed class.


\section{Geometric preliminaries}\label{sec:2}

We briefly recall some classical concepts from the geometry of Banach spaces.

\begin{defin}\label{def:2.1}
Let $X$ be a Banach space. The modulus of convexity \cite{Clark} of $X$ is the function $\delta_X : [0, 2] \to [0, 1]$ defined by
\begin{equation}
\delta_X(\epsilon) = \inf \Big\{ 1 - \Big\| \frac{x+y}{2} \Big\| : \|x\| \le 1, \|y\| \le 1, \|x-y\| \ge \epsilon \Big\}.
\end{equation}
\end{defin}

\smallskip 

This function quantifies, in a precise manner, the deviation of the norm from the strict convexity of the unit ball. In particular, positive values of $\delta_X(\epsilon)$ imply that points separated by at least $\epsilon$ have midpoints strictly inside $B_X$, providing a basis for compactness and convergence methods. Note that, for $d, \epsilon > 0$ and $x, y \in X$ with $\|x\| \le d$, $\|y\| \le d$ and $\|x-y\| \ge \epsilon$, one has $\left\| x+y \right\| \le 2d \left( 1 - \delta_X \left(\epsilon/d \right) \right)$. From the modulus of convexity, one can introduce a global quantity that measures the degree of convexity of the space.

\begin{defin}[\cite{Goebel,GK2}]\label{def:2.2}
The characteristic of convexity of a Banach space $X$ is defined by 
\[
\epsilon_0(X) = \sup \big\{ \epsilon \in [0, 2] : \delta_X(\epsilon) = 0 \big\}.
\]
\end{defin}

\smallskip 

The condition $\epsilon_0(X) < 1$ is pivotal in metric fixed point theory (see \cite{Goebel, james64, Turret82}). It is well known that spaces enjoying this condition are \textit{uniformly non-square} and, hence, reflexive (cf. \cite{Goebel,james64}). Furthermore, the class of uniformly convex spaces is strictly contained within the class of spaces satisfying $\epsilon_0(X) < 1$. In the sequel, we recall some key properties related to these functions that are used implicitly in this work. For instance, $\delta_X$ is continuous on $[0, 2)$ and increasing on $[\epsilon_0(X), 2]$ (cf. \cite[Lemma~5.1]{GK2}). Additionally, for examples of non-uniformly convex space with $\epsilon_0(X)<1$, see \cite[Section 4]{Casini-Maluta}. Finally, let us recall that a sequence $(x_n)$ in $X$ is said to be semi-normalized if $0< \inf_n\|x_n\|\leq \sup_n\|x_n\|<\infty$. It is called {\it basic} if it is a {\it Schauder} basis for its closed linear span $\llbracket  x_n \rrbracket$. 

\medskip 


\section{Main results}\label{sec:3}

The study of $\alpha$-H\"older continuous mappings was initiated by W. A. Kirk in 1998 \cite{Kirk} through minimal displacement estimates. More recent developments can be found in \cite{Bar1, Bar2}. Within this framework, however, the existence of fixed points becomes a central question.  In this paper, we address this problem by allowing H\"older exponents to oscillate and asymptotically converge to $1$. This approach enables us to move beyond displacement bounds and establish the FPP. Our main fixed point results are the following.

\subsection{Fixed point theorems}

\begin{theorem}\label{thm:main} Let $X$ be a Banach space such that $\epsilon_0(X) < 1$, and let $K \subset X$ be a nonempty, closed, bounded, and convex subset. If $T : K \to K$ is of asymptotically Hölder nonexpansive type, then there exists $x \in K$ such that $T^n x \to x$.
\end{theorem}

\begin{proof} For \(u\in K\), set
\[
d(u):=\limsup_{m\to\infty}\|u-T^m u\|.
\]
We shall construct a sequence \((x_n)\subset K\) such that
\[
d(x_{n+1})\leq\epsilon_0(X)d(x_n)
\]
and
\[
\|x_{n+1}-x_n\|\leq 2d(x_n).
\]
We first justify a normalization that will be used throughout the proof. Let \(D=\operatorname{diam}(K)\). If \(D=0\), the conclusion is immediate. Otherwise, replace \(K\) by \(D^{-1}K\) and \(T\) by the corresponding conjugate mapping. This transformation preserves the property of being of asymptotically H\"older nonexpansive type. Indeed,
if
\[
\omega_n(u):=\sup_{v\in K}\left(\|T^n u-T^n v\|-\|u-v\|^{\alpha_n}\right),
\]
then the error term for the rescaled mapping differs from \(D^{-1}\omega_n(u)\) by a quantity bounded in absolute value by
\[
D^{-1}\bigl|1-D^{1-\alpha_n}\bigr|D^{\alpha_n},
\]
which converges to zero because \(\alpha_n\to1\). We may therefore assume that $\operatorname{diam}(K)\leq1$. Fix \(x\in K\), and let
\[
\rho(x):=\inf_{z\in K}\limsup_{m\to\infty}\|z-T^m x\|
\]
be the asymptotic radius of the orbit \((T^m x)\) relative to \(K\). Since \(\varepsilon_0(X)<1\), the space \(X\) is reflexive. Hence \(K\) is weakly compact. Let
\[
r_x(z):=\limsup_{m\to\infty}\|z-T^m x\|,\qquad z\in K.
\]
It is easy to show that $r_x$ is $1$-Lipschitz and convex. Consequently, each of its sublevel sets is norm closed and convex, and hence weakly closed. For $j\geq 1$, set
\[
A_j(x):=\left\{ z\in K\colon r_x(z)\leq \rho(x)+ \frac1j\right\}.
\]
Each $A_j(x)$ is nonempty by the definition of $\rho(x)$, weakly closed in $K$, and the family $(A_j(x))$ is decreasing. Since $K$ is weakly compact, 
\[
\bigcap_{j=1}^{\infty}A_j(x)\neq\varnothing.
\]
Therefore, we may choose \(z\in K\) such that
\begin{equation}\label{eq:asymptotic-center}
\limsup_{m\to\infty}\|z-T^m x\|=\rho(x).
\end{equation}
Notice that \(x\) is an admissible competitor in the definition of \(\rho(x)\), and consequently
\begin{equation}\label{eq:rho-bounded-by-d}
\rho(x)\leq d(x).
\end{equation}
We claim that
\begin{equation}\label{eq:key-convexity-estimate}
d(z)\leq\epsilon_0(X)\rho(x).
\end{equation}
We first consider the case \(\rho(x)=0\). In this case $T^m x\to z$. For \(u\in K\), write
\[
\omega_n(u):=\sup_{v\in K}\left(\|T^n u-T^n v\|-\|u-v\|^{\alpha_n}\right).
\]
By hypothesis,
\[
\limsup_{n\to\infty}\omega_n(u)\leq0\qquad(u\in K).
\]
Choose integers \(m_n>n\) such that \(m_n-n\to\infty\). Then
\begin{align*}
\|z-T^n z\|
&\leq\|z-T^{m_n}x\|+\|T^n(T^{m_n-n}x)-T^n z\|\\
&\leq\|z-T^{m_n}x\|+\|T^{m_n-n}x-z\|^{\alpha_n}+\omega_n(z).
\end{align*}
Since \(T^m x\to z\), \(\alpha_n\to1\), and
\(\limsup_n\omega_n(z)\leq0\), it follows that
\[
\|z-T^n z\|\longrightarrow0.
\]
Thus \(d(z)=0\), and \eqref{eq:key-convexity-estimate} follows.

\smallskip 

Suppose now that \(\rho(x)>0\). If \(d(z)=0\), there is nothing to prove, so assume that \(d(z)>0\). Fix
\[
0<\varepsilon<\min\{\rho(x),d(z)\}.
\]
By \eqref{eq:asymptotic-center}, there exists \(N_\varepsilon\) such that
\begin{equation}\label{eq:center-tail-bound}
\|z-T^m x\|\leq\rho(x)+\varepsilon\qquad(m\geq N_\varepsilon).
\end{equation}
Choose an infinite set \(J_\varepsilon\subset\mathbb N\) such that
\begin{equation}\label{eq:d-subsequence}
\|z-T^jz\|\geq d(z)-\varepsilon\qquad(j\in J_\varepsilon).
\end{equation}
After removing finitely many indices, we may also assume that
\[
\omega_j(z)\leq\varepsilon\qquad(j\in J_\varepsilon).
\]
Fix \(j\in J_\varepsilon\). If \(m\geq N_\varepsilon+j\), then \(m-j\geq N_\varepsilon\), and hence
\begin{align*}
\|T^jz-T^m x\|
&=\|T^jz-T^j(T^{m-j}x)\|\\
&\leq\|z-T^{m-j}x\|^{\alpha_j}+\omega_j(z)\\
&\leq\bigl(\rho(x)+\varepsilon\bigr)^{\alpha_j}+\varepsilon.
\end{align*}
Since \(0\leq\rho(x),\varepsilon\leq1\) and \(0<\alpha_j\leq1\), concavity gives
\[
\bigl(\rho(x)+\varepsilon\bigr)^{\alpha_j}\leq\rho(x)^{\alpha_j}+\varepsilon^{\alpha_j}.
\]
Moreover, \(\varepsilon\leq\varepsilon^{\alpha_j}\). Therefore,
\begin{equation}\label{eq:second-radius}
\|T^jz-T^m x\|\leq R_{j,\varepsilon}:=\rho(x)^{\alpha_j}+2\varepsilon^{\alpha_j}.
\end{equation}
The same number also bounds the first distance. Indeed,
\[
\rho(x)+\varepsilon\leq\rho(x)^{\alpha_j}+\varepsilon^{\alpha_j}\leq R_{j,\varepsilon},
\]
and thus, by \eqref{eq:center-tail-bound},
\begin{equation}\label{eq:first-radius}
\|z-T^m x\|\leq R_{j,\varepsilon}.
\end{equation}
Set
\[
w_j:=\frac{z+T^jz}{2}\in K.
\]
By \eqref{eq:d-subsequence}, \eqref{eq:second-radius}, \eqref{eq:first-radius}, and the definition of the modulus of
convexity,
\[
\|w_j-T^m x\|\leq R_{j,\varepsilon}\left[1-\delta_X\left(\frac{d(z)-\varepsilon}{R_{j,\varepsilon}}\right)\right]
\]
for every sufficiently large \(m\). Since \(\rho(x)\) is the minimal asymptotic radius of the orbit, we obtain
\begin{equation}\label{eq:minimal-radius-inequality}
\rho(x)\leq R_{j,\varepsilon}\left[1-\delta_X\left(\frac{d(z)-\varepsilon}{R_{j,\varepsilon}}\right)\right].
\end{equation}
Let \(j\to\infty\) through \(J_\varepsilon\). Since
\(\alpha_j\to1\),
\[
R_{j,\varepsilon}\longrightarrow\rho(x)+2\varepsilon.
\]
Using the continuity of \(\delta_X\) on \([0,2)\), with the endpoint case obtained by monotonicity, inequality \eqref{eq:minimal-radius-inequality} yields
\[
    \rho(x)
    \leq
    \bigl(\rho(x)+2\varepsilon\bigr)
    \left[
        1-
        \delta_X\left(
            \frac{d(z)-\varepsilon}
                 {\rho(x)+2\varepsilon}
        \right)
    \right].
\]
Letting \(\varepsilon\downarrow0\), we get
\[
\rho(x)\leq\rho(x)\left[1-\delta_X\left(\frac{d(z)}{\rho(x)}\right)\right].
\]
Since \(\rho(x)>0\), it follows that
\[
\delta_X\left(\frac{d(z)}{\rho(x)}\right)=0.
\]
By the definition of the characteristic of convexity,
\[
\frac{d(z)}{\rho(x)}\leq\epsilon_0(X),
\]
which proves \eqref{eq:key-convexity-estimate}. We also have
\begin{align}
\|z-x\|
&\leq\limsup_{m\to\infty}\bigl(\|z-T^m x\|+\|T^m x-x\|\bigr)\notag\\
&\leq\rho(x)+d(x)\leq 2d(x).
\label{eq:center-distance}
\end{align}
Write \(q=\epsilon_0(X)<1\). Starting from an arbitrary \(x_0\in K\), construct inductively \(x_{n+1}\) as an asymptotic center
of the orbit of \(x_n\). From \eqref{eq:key-convexity-estimate}, \eqref{eq:rho-bounded-by-d}, and \eqref{eq:center-distance}, we obtain
\[
d(x_{n+1}) \leq q\,d(x_n)
\]
and
\[
\|x_{n+1}-x_n\|\leq 2d(x_n).
\]
Consequently,
\[
d(x_n)\leq q^n d(x_0)
\]
and
\[
\|x_{n+1}-x_n\|\leq 2q^n d(x_0).
\]
Thus \((x_n)\) is a Cauchy sequence. Since \(K\) is closed, there exists \(y\in K\) such that
\[
x_n\longrightarrow y.
\]
It remains to show that \(d(y)=0\). Fix \(n\). For every \(m\),
\begin{align*}
\|y-T^m y\|&\leq\|y-x_n\|+\|x_n-T^m x_n\|+
\|T^m x_n-T^m y\|.
\end{align*}
The asymptotically H\"older nonexpansive type condition gives
\[
\limsup_{m\to\infty}\|T^m x_n-T^m y\|\leq\lim_{m\to\infty}\|x_n-y\|^{\alpha_m}=\|x_n-y\|.
\]
Therefore,
\[
d(y)\leq d(x_n)+2\|x_n-y\|.
\]
Letting \(n\to\infty\), we get that \(d(y)=0\). Hence $\|T^m y-y\|\to 0$, proving the result. 
\end{proof}

\smallskip 

\begin{remark} If $T^N$ is continuous, then $T$ has a fixed point. Indeed, if \(T^nx\to x\) and \(T^N\) is continuous, then $T^{n+N}x=T^N(T^nx)\to T^Nx$. Since \(T^{n+N}x\to x\), it follows that \(T^Nx=x\). Hence $T^{kN+1}x=Tx$ for every \(k\), whereas \(T^{kN+1}x\to x\). Therefore \(Tx=x\).
\end{remark}

\begin{cor} Let $X$ and $K$ be as in Theorem \ref{thm:main}. Then every asymptotically Hölder nonexpansive mapping $T \colon K \to K$ possesses a fixed point.
\end{cor}

\vskip .1cm 

While these results significantly generalize \cite[Theorem 1]{Ki1}, the extent to which other classical results (such as those in \cite{Casini-Maluta,DBen1,MR}) can be extended to this framework remains a subject for future investigation. For the present, it is worth noting that a compact version of Theorem \ref{thm:main} follows directly from the arguments in \cite{Ki1}:

\vskip .1cm 

\begin{theorem}\label{thm:maincompact} Let $K$ be a nonempty, compact, and convex subset of a Banach space $X$. Suppose $T \colon K \to K$ is of asymptotically H\"older nonexpansive type and has a continuous iterate $T^N$ for some $N \in \mathbb{N}$. Then $T$ has a fixed point in $K$.
\end{theorem}

\subsection{Fixed-point-free constructions}

The next result underlines the relevance of Theorem \ref{thm:maincompact}.

\begin{theorem}\label{thm:main4} Let $X$ be an infinite dimensional Banach space. Then there exists a compact convex set $K\subset X$ and a fixed-point-free affine mapping $T\colon K\to K$ which is of asymptotically H\"older nonexpansive type and none of its positive iterates $T^n$ is continuous. 
\end{theorem}

\begin{proof}
Although the core argument is implicitly contained in \cite[Theorem 4.17]{Bar2}, we outline its main steps here for the sake of self-containment and the reader's convenience. Let $(x_n)_{n=1}^\infty$ be a normalized basic sequence and denote by $\mathsf{K}_{\mathsf{m}}$ its basic constant. Choose a decreasing null sequence $(\gamma_n)_{n=1}^\infty$ in $(0,1)$ satisfying
\[
\max\Bigg( 3\sum_{i=n}^\infty \gamma_i\,,\, 2\mathsf{K}_{\mathsf{m}}\sum_{i=1}^\infty \frac{ \gamma_{i+n}}{\gamma_i}\Big)\leq \frac{1}{2^{n+1}}, \quad \forall n\in\mathbb{N}.
\]
For each $n \in \mathbb{N}$, define $u_n=\sum_{i=1}^n \gamma_i x_i$. The sequence $(u_n)_{n=1}^\infty$ converges in norm to $u_0=\sum_{i=1}^\infty \gamma_i x_i$ and is semi-normalized (cf. \cite[p. 17]{Bar2}). Furthermore, by suitably defining the functional sequence $(u^*_n)_{n=1}^\infty$, one readily verifies that $(u_n, u^*_n)_{n=1}^\infty$ forms a biorthogonal system on the closed linear span $\llbracket x_n \rrbracket$. Now, consider the set
\[
K=\Bigg\{ x:=t_0 u_0 + \sum_{n=1}^\infty t_n u_n \colon t_n\geq 0\,\,\forall n\in\mathbb{N}\cup\{0\},\,\, \sum_{n=0}^\infty t_n\leq 1\Bigg\}.
\]
We observe that $K$ is a compact and convex set (see \cite[Claim 1]{Bar2}). Let us now introduce the affine mapping $T\colon K\to K$ defined by
\[
T(x) =\Bigg( 1 - \sum_{n=1}^\infty t_n\Bigg)u_1 + \sum_{n=1}^\infty t_n u_{n+1}.
\]
It was shown in \cite{Bar2} that $T$ is fixed point free. Thus by Schauder's fixed point theorem it cannot be continuous. Moreover, for any $k\in\mathbb{N}$, a straightforward calculation yields
\[
T^k(x)= \Bigg( 1 - \sum_{n=1}^\infty t_n\Bigg) u_k+ \sum_{n=1}^\infty t_n u_{n+k}.
\]
Direct computation shows that each iterate $T^k$ is likewise affine and fixed point free; hence, none of them can be continuous. Take arbitrary points $x, y\in K$, say $x=\sum_{n=0}^\infty t_n u_n$ and $y=\sum_{n=0}^\infty s_n u_n$, and set $a_n=t_n- s_n$. Then we have
\[
T^k(x) - T^k(y)= \sum_{n=1}^\infty\Bigg(a_0 + \sum_{i=n}^\infty a_i\Bigg)\gamma_{n+k} x_{n+k} - a_0\sum_{n=1}^\infty \gamma_{n+k} x_{n+k},
\]
from which it follows that
\[
\| T^k(x) - T^k(y)\|\leq\Bigg( 2\mathsf{K}_{\mathsf{m}}\sum_{n=1}^\infty \frac{\gamma_{n+k}}{\gamma_{n}}\Bigg) \|x - y\| + \frac{1}{2^k}. 
\]
In particular, for any increasing sequence $\alpha_k\to 1$ in $(0,1)$, we have
\[
\limsup_{k\to\infty} \sup_{y\in K}\Big( \|T^k(x) - T^k(y)\| - \|x-y\|^{\alpha_k}\Big)\leq 0.
\]
\end{proof}

\vskip .1cm 

The following construction shows that the geometric assumption in Theorem~\ref{thm:main} cannot be removed without imposing additional restrictions. 

\begin{theorem}\label{thm:main2} Let $X$ be a Banach space containing an isomorphic copy of $c_0$. Then there exists a bounded, closed, and convex subset $K$ such that, for any increasing sequence $\alpha_n \to 1$, there exists a fixed-point-free asymptotically $(\alpha_n)$-Hölder-nonexpansive mapping $T \colon K \to K$.
\end{theorem}

\begin{proof}
Since \(X\) contains an isomorphic copy of 
\(c_0\), the distortion argument used in \cite{DLT1} provides a basic sequence \((e_n)\subset X\) and a decreasing null sequence
\((\delta_n)\subset(0,\infty)\) such that
\begin{equation}\label{eq:c0-tail-estimate-corrected}
    \sup_{j\geq m}|a_j|\leq \left\|\sum_{j=m}^{\infty}a_j e_j\right\|\leq (1+\delta_m)\sup_{j\geq m}|a_j|
\end{equation}
for every \((a_j)\in c_0\) and every \(m\geq1\). Fix this sequence and set
\[
K:=\left\{\sum_{j=1}^{\infty}t_j e_j:(t_j)\in c_0,\quad 0\leq t_j\leq1\right\}.
\]
Using \eqref{eq:c0-tail-estimate-corrected} one easily shows that \(K\) is closed, bounded, and convex.  Now let $(\alpha_n)\subset(0,1)$ be an arbitrary increasing sequence
such that $\alpha_n\to1$. Choose a sequence \((B_j)_{j\geq2}\subset(0,1)\) such that
\begin{equation}\label{eq:positive-B-product}
    \prod_{j=2}^{\infty}B_j>0.
\end{equation}
For \(j\geq2\), put
\[
    p_j:=\frac{\alpha_j}{\alpha_{j-1}}\ge 1
\]
and
\begin{equation}\label{eq:corrected-coefficients}
    c_j
    :=
    \frac{\alpha_{j-1}}{\alpha_j}B_j
    =
    \frac{B_j}{p_j}.
\end{equation}
Then \(0<c_j<1\) and
\begin{equation}\label{eq:derivative-cancellation}
    c_jp_j=B_j<1.
\end{equation}
Define \(T\colon K\to K\) by
\begin{equation}\label{eq:c0-corrected-map}
    T\left(\sum_{j=1}^{\infty}t_j e_j\right)
    =
    e_1+t_1^{\alpha_1}e_2
    +
    \sum_{j=2}^{\infty}
    c_j t_j^{p_j}e_{j+1}.
\end{equation}
This mapping is well defined. Indeed, if \((t_j)\in c_0\) and
\(0\leq t_j\leq1\), then
\[
    0\leq c_jt_j^{p_j}\leq t_j^{p_j}\leq t_j,
\]
so
\[
c_jt_j^{p_j}\longrightarrow0.
\]
All the coordinates of \(T(x)\) belong to \([0,1]\), and hence \(T(x)\in K\). We next show that \(T\) has no fixed point. Suppose, to the contrary, that $x=\sum_{j=1}^{\infty}t_j e_j\in K$ satisfies \(T(x)=x\). Comparing coordinates in \eqref{eq:c0-corrected-map}, we obtain
\[
t_1=1,\qquad t_2=t_1^{\alpha_1}=1,
\]
and
\begin{equation}\label{eq:fixed-coordinate-recursion}
t_{j+1}=c_jt_j^{p_j},\qquad j\geq2.
\end{equation}
In particular, every \(t_j\) is strictly positive. Set
\[
s_j:=-\log t_j,\qquad j\geq2.
\]
Then \(s_2=0\), and \eqref{eq:fixed-coordinate-recursion} gives
\[
s_{j+1}=-\log c_j+\frac{\alpha_j}{\alpha_{j-1}}s_j.
\]
Dividing by \(\alpha_j\), we get
\[
\frac{s_{j+1}}{\alpha_j}=\frac{s_j}{\alpha_{j-1}}+\frac{-\log c_j}{\alpha_j}.
\]
Iteration yields
\begin{equation}\label{eq:logarithmic-fixed-recursion}
\frac{s_{n+1}}{\alpha_n}=\sum_{j=2}^{n}\frac{-\log c_j}{\alpha_j}.
\end{equation}
By \eqref{eq:corrected-coefficients},
\[
-\log c_j=\log\left(\frac{\alpha_j}{\alpha_{j-1}}\right)-\log B_j.
\]
Since \(\alpha_j\geq\alpha_1>0\),
\begin{align*}
\sum_{j=2}^{\infty}\frac{-\log c_j}{\alpha_j}
&\leq\frac1{\alpha_1}\sum_{j=2}^{\infty}\left[\log\left(\frac{\alpha_j}{\alpha_{j-1}}\right)-\log B_j\right]\\
&=\frac1{\alpha_1}\left[\log\left(\frac1{\alpha_1}\right)-\log\left(\prod_{j=2}^{\infty}B_j\right)\right]<\infty.
\end{align*}
It follows from \eqref{eq:logarithmic-fixed-recursion} that \((s_n)\) is bounded. Consequently,
\[
\inf_{n\geq2}t_n=\inf_{n\geq2}e^{-s_n}>0,
\]
contradicting \((t_n)\in c_0\). Therefore \(T\) is fixed point free. We now proceed to prove $T$ is asymptotically \((\alpha_n)\)-H\"older nonexpansive. Introduce the scalar maps
\[
\varphi_1(t):=t^{\alpha_1}
\]
and
\[
\varphi_j(t):=c_jt^{p_j},\qquad j\geq2,
\]
on \([0,1]\). For \(j\geq2\),
\[
\varphi_j'(t)=c_jp_jt^{p_j-1},
\]
and therefore, by \eqref{eq:derivative-cancellation},
\[
0\leq\varphi_j'(t)\leq B_j<1,\qquad 0\leq t\leq1.
\]
Thus
\begin{equation}\label{eq:scalar-Lipschitz}
    |\varphi_j(s)-\varphi_j(t)|
    \leq
    B_j|s-t|
    \leq
    |s-t|.
\end{equation}
For \(k\geq2\), the contribution of the \(k\)-th coordinate of \(x\) to \(T^n x\) is obtained by composing
\[
\varphi_k,\varphi_{k+1},\ldots,\varphi_{k+n-1}.
\]
Hence \eqref{eq:scalar-Lipschitz} gives
\begin{equation}\label{eq:tail-coordinate-bound}
\left|(\varphi_{k+n-1}\circ\cdots\circ\varphi_k)(t)-(\varphi_{k+n-1}\circ\cdots\circ\varphi_k)(s)\right|\leq|t-s|.
\end{equation}
For the first coordinate, the composition has the form
\[
(\varphi_n\circ\cdots\circ\varphi_2\circ\varphi_1)(t)=C_n t^{\alpha_n}
\]
for some \(C_n\in(0,1]\). Since
\[
|t^{\alpha_n}-s^{\alpha_n}|\leq |t-s|^{\alpha_n}, \qquad s,t\in[0,1],
\]
we obtain
\begin{equation}\label{eq:first-coordinate-bound}
    \left|
       (\varphi_n\circ\cdots\circ\varphi_1)(t)
       -
       (\varphi_n\circ\cdots\circ\varphi_1)(s)
    \right|
    \leq
    |t-s|^{\alpha_n}.
\end{equation}
Let $x=\sum_{j=1}^{\infty}t_j e_j$, $y=\sum_{j=1}^{\infty}s_j e_j$ belong to \(K\). The first \(n\) coordinates of \(T^nx\) and \(T^ny\) coincide. By \eqref{eq:tail-coordinate-bound} and \eqref{eq:first-coordinate-bound},
\[
\sup_{j\geq n+1}|(T^nx-T^ny)_j|\leq\sup_{k\geq1}|t_k-s_k|^{\alpha_n}.
\]
Here we used \(0\leq|t_k-s_k|\leq1\), which implies $|t_k-s_k|\leq |t_k-s_k|^{\alpha_n}$. Applying \eqref{eq:c0-tail-estimate-corrected} to the tail beginning at the \((n+1)\)-st coordinate, we obtain
\begin{align*}
    \|T^nx-T^ny\|
    &\leq
    (1+\delta_{n+1})
    \sup_{j\geq n+1}|(T^nx-T^ny)_j|\\
    &\leq
    (1+\delta_{n+1})
    \left(
        \sup_{k\geq1}|t_k-s_k|
    \right)^{\alpha_n}.
\end{align*}
Consequently, using the lower estimate in \eqref{eq:c0-tail-estimate-corrected}, applied with \(m=1\), gives
\[
\|T^nx-T^ny\|\leq (1+\delta_{n+1})\|x-y\|^{\alpha_n}.
\]
Since \(\delta_{n+1}\to0\), this completes the proof.
\end{proof}

The corresponding closed-set problem for spaces containing \(\ell_1\) appears to be more delicate. We obtain the following
partial result.

\begin{theorem}\label{thm:main3} Let $X$ be a Banach space containing an isomorphic copy of $\ell_1$. Then there exists a bounded and convex subset $K$ such that, for any increasing sequence $\alpha_n \to 1$ in $(0,1)$, there exists an affine, fixed-point-free mapping $T \colon K \to K$ that is uniformly asymptotically $(\alpha_n)$-Hölder 1-Lipschitz.
\end{theorem}

\begin{proof}
Let $\eta\in (0,1)$ be fixed. In view of our assumptions and James’s distortion theorem (cf. \cite{DLT1}), there exists a basic sequence $(e_i)$ in $X$ such that
\begin{equation}\label{eqn:7}
\eta\sum_{n=1}^\infty |a_n|\leq \Bigg\| \sum_{n=1}^\infty a_n e_n\Bigg\|\leq \sum_{n=1}^\infty |a_n|,
\end{equation}
for all $(a_n)_{n=1}^\infty \in \ell_1$. 

Let $(\lambda_n) \subset (0,1)$ be a decreasing null sequence such that $\sum_{n=1}^\infty \lambda_n \leq \eta$ and $\lambda_{n+k} \leq \lambda_n \lambda_k$ for all $n, k \in \mathbb{N}$.  Define
\[
K=\Bigg\{ x=\sum_{n=1}^\infty t_n f_n \colon t_n\geq 0,\,\sum_{n=1}^\infty t_n=1,\, f_n=\lambda_n e_n\Bigg\}.
\]
Then $K$ is clearly bounded and convex. Now define $T \colon K \to K$ to be the right-shift mapping induced by $(f_n)$, given by $T(f_n) = f_{n+1}$. It is easy to see that $T$ has no fixed point. Fix an increasing sequence $(\alpha_n) \subset (0,1)$ converging to $1$. For $n\in\mathbb{N}$, set $p_n= (1-\alpha_n)^{-1}$ and take $q_n>1$ so that $1/p_n+1/q_n=1$. For any $k \in \mathbb{N}$ and points $x = \sum t_n f_n, y = \sum s_n f_n \in K$, combining Hölder's inequality with \eqref{eqn:7} yields
\[
\begin{split}
\|T^k(x) - T^k(y)\|
&\leq \Bigg(\sum_{n=1}^\infty |t_n - s_n|^{q_k}\lambda_{n+k}\Bigg)^{\alpha_k}\cdot\Bigg( \sum_{n=1}^\infty \lambda_{n+k}\Bigg)^{1-\alpha_k}\\[2mm]
&\leq \Bigg(\lambda_k \sum_{n=1}^\infty |t_n - s_n|\lambda_n\Bigg)^{\alpha_k}\\[2mm]
&\leq \Bigg( \eta\sum_{n=1}^\infty |t_n - s_n|\lambda_n\Bigg)^{\alpha_k}\leq 
\|x -y\|^{\alpha_k}. 
\end{split}
\]
\end{proof}

\subsection{Asymptotic nonexpansiveness on $B_{\ell_1}$}

It is worth noting that the set \(K\) in Theorem \ref{thm:main3} is not closed. By contrast, on the closed unit ball of \(\ell_1\), the norm-splitting property below yields a strong rigidity phenomenon, showing that the existence of an approximate fixed point sequence already forces the existence of a fixed point.

\begin{lemma}
\label{lem:splitting-l1}
Let $(x_j)$ be a bounded sequence in $\ell_1$ such that
\[
x_j \xrightarrow{\sigma(\ell_1,c_0)} z,
\]
and suppose that
\[
\lim_{j\to\infty}\|x_j-z\|_1=r.
\]
Then, for every $y\in\ell_1$,
\[
\lim_{j\to\infty}\|x_j-y\|_1=r+\|z-y\|_1.
\]
\end{lemma}

\begin{proof}
Set
\[
u_j=x_j-z\qquad\text{and}\qquad v=z-y.
\]
Then $u_j\to0$ coordinatewise,
\[
\|u_j\|_1\longrightarrow r,
\]
and $x_j-y=u_j+v$. It is therefore enough to prove that
\[
\|u_j+v\|_1-\|u_j\|_1\longrightarrow\|v\|_1.
\]
Fix $\varepsilon>0$. Choose a finite set $F\subset\mathbb N$ such that
\[
\sum_{k\notin F}|v(k)|<\varepsilon.
\]
Since $u_j(k)\to0$ for every $k\in F$, we have
\[
\sum_{k\in F}\bigl(|u_j(k)+v(k)|-|u_j(k)|\bigr)\longrightarrow\sum_{k\in F}|v(k)|.
\]
On the other hand, by the triangle inequality,
\[
\left|\sum_{k\notin F}\bigl(|u_j(k)+v(k)|-|u_j(k)|\bigr)\right|\le\sum_{k\notin F}|v(k)|<\varepsilon.
\]
Consequently,
\[
\limsup_{j\to\infty}\left|\|u_j+v\|_1-\|u_j\|_1-\sum_{k\in F}|v(k)|\right|\le\varepsilon.
\]
Moreover,
\[
\left|\|v\|_1-\sum_{k\in F}|v(k)|\right|<\varepsilon.
\]
Hence,
\[
\limsup_{j\to\infty}\left|\|u_j+v\|_1-\|u_j\|_1-\|v\|_1\right|\le 2\varepsilon.
\]
Since $\varepsilon>0$ is arbitrary,
\[
\|u_j+v\|_1-\|u_j\|_1\longrightarrow \|v\|_1.
\]
Finally, since $\|u_j\|_1\to r$, we conclude that
\[
\|x_j-y\|_1=\|u_j+v\|_1 \longrightarrow r+\|z-y\|_1.
\]
\end{proof}

This lemma allows us to prove the following characterization.

\begin{theorem}
\label{thm:l1-gauge-criterion}
Let \(T\colon B_{\ell_1}\to B_{\ell_1}\), and suppose that there exists
a sequence of continuous functions $\varphi_n\colon[0,2]\to[0,\infty)$, $n\geq1$, such that $\varphi_n(0)=0$ for $n\geq 1$,
\begin{equation}\label{eq:gauge-pointwise-identity}
    \varphi_n(t)\longrightarrow t
    \qquad(t\in[0,2]),
\end{equation}
and
\begin{equation}\label{eq:gauge-iterate-estimate}
    \|T^nx-T^ny\|_1
    \leq
    \varphi_n(\|x-y\|_1)
\end{equation}
for all \(x,y\in B_{\ell_1}\) and \(n\geq1\). Then the following conditions are equivalent:
\begin{itemize}
    \item[(i)] \(T\) has a fixed point;

    \item[(ii)] \(T\) has an approximate fixed point sequence, that is, there
    exists \((x_j)\subset B_{\ell_1}\) such that $\|Tx_j-x_j\|_1\to 0$.
\end{itemize}
\end{theorem}

\begin{proof}
The implication $(i)\Longrightarrow(ii)$ is immediate. We proceed to prove \((ii)\Rightarrow(i)\).

Let \((x_j)\subset B_{\ell_1}\) satisfy
\[
d_j:=\|Tx_j-x_j\|_1\longrightarrow0.
\]
Since \(\ell_1=c_0^*\) and \(c_0\) is separable, \(B_{\ell_1}\) is
sequentially compact in the weak-star topology. Passing to a
subsequence, we may assume that there exists \(z\in B_{\ell_1}\) such
that
\[
x_j\xrightarrow{\sigma(\ell_1,c_0)}z.
\]
Passing to a further subsequence, we may also assume that
\begin{equation}\label{eq:gauge-r-limit}
\|x_j-z\|_1\longrightarrow r
\end{equation}
for some \(r\in[0,2]\). We first show that, for every fixed \(m\geq1\),
\begin{equation}\label{eq:gauge-approximate-iterate}
\|T^mx_j-x_j\|_1\longrightarrow0.
\end{equation}
Define $\varphi_0(t):=t$ for $0\leq t\leq2$. By the triangle inequality and \eqref{eq:gauge-iterate-estimate},
\begin{align*}
    \|T^mx_j-x_j\|_1
    &\leq
    \sum_{i=0}^{m-1}
    \|T^{i+1}x_j-T^ix_j\|_1=\sum_{i=0}^{m-1}\|T^i(Tx_j)-T^i(x_j)\|_1\\
    &\leq
    \sum_{i=0}^{m-1}\varphi_i(d_j).
\end{align*}
For fixed \(m\), each function \(\varphi_i\), \(0\leq i<m\), is continuous at \(0\) and satisfies \(\varphi_i(0)=0\). Since
\(d_j\to0\), it follows that
\[
\sum_{i=0}^{m-1}\varphi_i(d_j)\longrightarrow0,
\]
which proves \eqref{eq:gauge-approximate-iterate}. Fix \(m\geq1\). From \eqref{eq:gauge-iterate-estimate},
\begin{align*}
\|x_j-T^mz\|_1
&\leq \|x_j-T^mx_j\|_1+\|T^mx_j-T^mz\|_1\\
&\leq \|x_j-T^mx_j\|_1+\varphi_m(\|x_j-z\|_1).
\end{align*}
Using \eqref{eq:gauge-approximate-iterate},
\eqref{eq:gauge-r-limit}, and the continuity of \(\varphi_m\), we
obtain
\begin{equation}\label{eq:gauge-upper-limit}
\limsup_{j\to\infty}\|x_j-T^mz\|_1\leq \varphi_m(r).
\end{equation}
On the other hand, the norm-splitting property of \(\ell_1\),
Lemma~\ref{lem:splitting-l1}, applied with \(y=T^mz\), gives
\begin{equation}\label{eq:gauge-splitting}
\lim_{j\to\infty}\|x_j-T^mz\|_1 =r+\|T^mz-z\|_1.
\end{equation}
Combining \eqref{eq:gauge-upper-limit} and
\eqref{eq:gauge-splitting}, we obtain
\begin{equation}\label{eq:gauge-orbit-bound}
0\leq \|T^mz-z\|_1\leq \varphi_m(r)-r.
\end{equation}
By \eqref{eq:gauge-pointwise-identity},
\[
\varphi_m(r)\longrightarrow r \qquad(m\to\infty).
\]
It follows from \eqref{eq:gauge-orbit-bound} that
\begin{equation}\label{eq:gauge-orbit-converges}
T^mz\longrightarrow z
\end{equation}
in norm. Finally, \(T\) is continuous. Indeed,
\[
\|Tx-Ty\|_1\leq \varphi_1(\|x-y\|_1),
\]
and \(\varphi_1(t)\to0\) as \(t\to0\). Hence
\[
T^{m+1}z=T(T^mz)\longrightarrow Tz.
\]
On the other hand, \eqref{eq:gauge-orbit-converges}, with \(m+1\) in place of \(m\), gives $T^{m+1}z\longrightarrow z$. By uniqueness of the norm limit, $Tz=z$. Thus \(T\) has a fixed point.
\end{proof}

\begin{cor}\label{cor:l1-variable-holder}
Let \(T\colon B_{\ell_1}\to B_{\ell_1}\) satisfy
\[
\|T^nx-T^ny\|_1\leq (1+\kappa_n)\|x-y\|_1^{\alpha_n},
\]
where \((\alpha_n)\subset (0,1]$, $(\kappa_n)\subset[0,\infty)\) and \(\alpha_n\to1\), \(\kappa_n\to 0\). Then \(T\) has
a fixed point if and only if it has an approximate fixed point
sequence.
\end{cor}

\begin{proof}
Apply Theorem~\ref{thm:l1-gauge-criterion} with
\[
\varphi_n(t):=(1+\kappa_n)t^{\alpha_n},\qquad 0\leq t\leq2.
\]
For every fixed \(t\in[0,2]\), $(1+\kappa_n)t^{\alpha_n}\to t$.
\end{proof}

It is known that affine self-maps on convex bounded sets have approximate fixed point sequences. Consequently, we have the following corollary. 

\begin{cor} $B_{\ell_1}$ has the fixed point property for affine asymptotically H\"older nonexpansive mappings. 
\end{cor}

\begin{remark} It is worth noting that Karlovitz \cite{Karlovitz} proved that $B_{\ell_1}$ has the FPP, whereas Khamsi \cite{Khamsi} showed that bounded hyperconvex metric spaces have the approximate fixed point property for asymptotically nonexpansive mappings. To the best of our knowledge, it remains an open question whether or not closed, bounded, convex sets in Banach spaces share this latter property.
\end{remark}

\subsection{Asymptotic nonexpansivity away from the diagonal}

Throughout this subsection, \(X\) denotes a Banach space and \(K\subset X\) is a nonempty, closed, bounded, and convex set. For a mapping \(T\colon K\to K\), we write
\[
\operatorname{Lip}(T)=\sup_{\substack{x,y\in K\\x\neq y}}\frac{\|Tx-Ty\|}{\|x-y\|}.
\]

The following condition isolates the amount of asymptotic
nonexpansivity that is needed away from a shrinking neighborhood of the
diagonal.

\begin{defin}
\label{def:off-diagonal}
A mapping \(T\colon K\to K\) is said to be \emph{asymptotically nonexpansive away from the diagonal} if there exist
null positive sequences $(\delta_n)$ and $(\eta_n)$ such that
\[
\|T^nx-T^ny\|\le(1+\eta_n)\|x-y\|
\]
whenever \(x,y\in K\) satisfy
\[
\|x-y\|\ge\delta_n.
\]
\end{defin}

The main observation is that uniform Lipschitz control near the diagonal can be combined with asymptotic nonexpansivity away from the diagonal by means of a suitable sequence of H\"older exponents.

\begin{theorem}
\label{thm:off-diagonal-holder}
Let \(K\subset X\) be nonempty, closed, bounded, and convex, with $\operatorname{diam}(K)\le1$. Suppose that \(T\colon K\to K\) satisfies
\[
M:=\sup_{n\ge1}\operatorname{Lip}(T^n)<\infty
\]
and is asymptotically nonexpansive away from the diagonal. Then there
exist sequences $(\alpha_n)\subset(0,1)$, $(\kappa_n)\subset(1,\infty)$, such that $\alpha_n\to 1$, $\kappa_n\to 1$, and
\begin{equation}\label{eqn:thm3.9}
\|T^nx-T^ny\|\le\kappa_n\|x-y\|^{\alpha_n}
\end{equation}
for all \(x,y\in K\) and \(n\ge1\). Consequently, if \(T\) is fixed point free, then \(K\) fails the fixed point property for asymptotically H\"older nonexpansive mappings.
\end{theorem}

\begin{proof}Set
\[
L_n:=\operatorname{Lip}(T^n)\qquad\text{and}\qquad M:=\sup_{n\geq1}L_n<\infty.
\]
First suppose that \(M\leq1\). Then every iterate \(T^n\) is nonexpansive. In particular, since $\diam(K)\leq 1$, one readily shows that $T$ satisfies \eqref{eqn:thm3.9} for all sequences $(\alpha_n)$ and $(\kappa_n)$ as in the statement of the theorem. This proves the result when \(M\leq1\). 

Assume now that \(M>1\). Since \(\delta_n\to0\), there exists \(n_0\in\mathbb N\) such that
\begin{equation}\label{eq:delta-small-eventually}
0<\delta_n<\frac1M\qquad(n\geq n_0).
\end{equation}
For every \(n\geq n_0\), define
\begin{equation}\label{eq:alpha-from-original-delta}
1-\alpha_n:=\frac{\log M}{\log(1/\delta_n)}.
\end{equation}
Condition \eqref{eq:delta-small-eventually} gives
\[
\log(1/\delta_n)>\log M>0,
\]
and hence
\[
0<1-\alpha_n<1.
\]
Thus
\[
0<\alpha_n<1\qquad(n\geq n_0).
\]
Furthermore, since \(\delta_n\to0\),
\[
\log(1/\delta_n)\longrightarrow\infty,
\]
so that
\[
\alpha_n\longrightarrow1.
\]
The definition of \(\alpha_n\) also yields the identity
\begin{equation}\label{eq:delta-absorbs-M}
\delta_n^{\,1-\alpha_n}=\exp\left(-\log(1/\delta_n)\frac{\log M}{\log(1/\delta_n)}\right)=\frac1M.
\end{equation}
For \(n\geq n_0\), set
\begin{equation}\label{eq:kappa-off-diagonal}
\kappa_n:=\max\left\{1+\eta_n,\,1+\frac1n\right\}.
\end{equation}
Then \(\kappa_n>1\) and \(\kappa_n\to1\). Fix \(n\geq n_0\) and \(x,y\in K\), and put
\[
d:=\|x-y\|.
\]
The case \(d=0\) is immediate, so assume that \(d>0\). We distinguish two cases. If \(0<d<\delta_n\), then uniform Lipschitzianity gives
\[
\|T^n x-T^n y\|\leq L_n d\leq Md.
\]
Since \(1-\alpha_n>0\) and \(d<\delta_n\),
\[
d^{1-\alpha_n}\leq\delta_n^{1-\alpha_n}=\frac1M
\]
by \eqref{eq:delta-absorbs-M}. Consequently,
\[
Md=M d^{1-\alpha_n}d^{\alpha_n}\leq d^{\alpha_n}.
\]
Hence
\begin{equation}\label{eq:near-diagonal-estimate}
\|T^n x-T^n y\|\leq d^{\alpha_n}\leq \kappa_n d^{\alpha_n}.
\end{equation}
If \(d\geq\delta_n\), then asymptotic nonexpansiveness away from the diagonal gives
\[
\|T^n x-T^n y\|\leq (1+\eta_n)d.
\]
Because \(d\leq1\) and \(0<\alpha_n<1\), we have
\[
d\leq d^{\alpha_n}.
\]
Therefore,
\begin{equation}\label{eq:far-from-diagonal-estimate}
\|T^n x-T^n y\|\leq (1+\eta_n)d^{\alpha_n}\leq \kappa_n d^{\alpha_n}.
\end{equation}
Notice that the boundary case \(d=\delta_n\) is covered by the off-diagonal hypothesis; it could also be handled by \eqref{eq:delta-absorbs-M}. It remains to define \(\alpha_n\) and \(\kappa_n\) for the finitely many indices \(1\leq n<n_0\). Choose arbitrary numbers
\[
\alpha_n\in(0,1),\qquad 1\leq n<n_0,
\]
and set, for example,
\[
\kappa_n:=M+1.
\]
Since \(0\leq d\leq1\), we have \(d\leq d^{\alpha_n}\), and hence
\[
\|T^n x-T^n y\|\leq L_n d\leq Md\leq (M+1)d^{\alpha_n}=\kappa_n d^{\alpha_n}.
\]
Changing finitely many terms does not affect the limits
\[
\alpha_n\to1 \qquad\text{and}\qquad\kappa_n\to1.
\]
Combining the estimates above, we conclude that
\[
\|T^n x-T^n y\|\leq\kappa_n\|x-y\|^{\alpha_n}
\]
for every \(x,y\in K\) and every \(n\geq1\).
\end{proof}

\begin{remark}
\label{rem:distortion-diagonal}
Theorem~\ref{thm:off-diagonal-holder} shows that the relevant issue is not whether the iterates are asymptotically nonexpansive at every scale. A fixed amount of expansion is allowed near the diagonal, provided that
the region where this expansion occurs shrinks sufficiently fast. More precisely, the condition permits $\operatorname{Lip}(T^n)\le M$ with a fixed \(M>1\), while requiring the distortion to be at most \(1+\eta_n\) on pairs whose distance is at least \(\delta_n\). The H\"older exponent absorbs the factor \(M\) on distances below \(\delta_n\).
\end{remark}

\begin{proposition}
\label{prop:holder-implies-type}
Let \(K\subset X\) be bounded, and suppose that
\[
\|T^nx-T^ny\|\le\kappa_n\|x-y\|^{\alpha_n}
\]
for all \(x,y\in K\), where $(\alpha_n)\subset (0,1]$ and $(\kappa_n)\subset (0,\infty)$ satisfy
\[
\kappa_n\longrightarrow1\qquad\text{and}\qquad\alpha_n\longrightarrow1.
\]
Then \(T\) is of asymptotically nonexpansive type, that is,
\[
\limsup_{n\to\infty}\sup_{y\in K}\bigl(\|T^nx-T^ny\|-\|x-y\|\bigr)\le0
\]
for every \(x\in K\).
\end{proposition}

\begin{proof}
Let $D=\operatorname{diam}(K)$. It suffices to prove that
\[
\sup_{0\le t\le D}\bigl(\kappa_n t^{\alpha_n}-t\bigr)\longrightarrow0.
\]
Fix \(\varepsilon>0\). Choose \(\delta>0\) sufficiently small that
\[
4\delta<\varepsilon.
\]
Since \(\alpha_n\to1\), for all sufficiently large \(n\),
\[
\delta^{\alpha_n}<2\delta.
\]
Also, since \(\kappa_n\to1\), we may assume that \(\kappa_n<2\). Hence, for \(0\le t\le\delta\),
\[
\kappa_nt^{\alpha_n}-t\le 2\delta^{\alpha_n}<4\delta.
\]
By choosing \(\delta\) smaller initially, this term can be made less than \(\varepsilon\). On the compact interval \([\delta,D]\), the functions
\[
t\longmapsto \kappa_nt^{\alpha_n}
\]
converge uniformly to \(t\). Therefore,
\[
\sup_{\delta\le t\le D} \left|\kappa_nt^{\alpha_n}-t \right|\longrightarrow0.
\]
This proves the desired uniform convergence.
\end{proof}


\section{The fixed point problem for Lin's renorming of \(\ell_1\)}
\label{sec:4}

Let \((\gamma_n)\subset(0,1)\) be an increasing sequence such that \(\gamma_n\to1\), and consider the equivalent norm on 
\(\ell_1\) given by
\begin{equation}\label{eq:Lin-norm}
\|x\|_{\mathrm L}=\sup_{m\geq1}\gamma_m\sum_{j=m}^{\infty}|x_j|,\qquad x=(x_j)\in\ell_1.
\end{equation}
Indeed,
\begin{equation}\label{eq:Lin-equivalence}
\gamma_1\|x\|_1\leq\|x\|_{\mathrm L}\leq\|x\|_1.
\end{equation}

Lin's fixed point theorem \cite[Example 1, Theorem 4]{Lin2} (see also \cite{Lin1} when $\gamma_n=\frac{8^n}{1+8^n}$) states that the renormed space \((\ell_1,\|\cdot\|_{\mathrm L})\) has the FPP for nonexpansive mappings. This makes Lin's renorming a natural test case for the difference between nonexpansiveness and asymptotic H\"older nonexpansiveness. 

\begin{problem}\label{prob:lin-holder}
Do there exist a nonempty, closed, bounded, and convex subset \(K\subset(\ell_1,\|\cdot\|_{\mathrm L})\), a fixed-point-free mapping \(T\colon K\to K\), and sequences
\[
\alpha_n\in(0,1),\qquad\kappa_n\in(1,\infty),\qquad\alpha_n\to1,\qquad\kappa_n\to1,
\]
such that
\begin{equation}\label{eq:variable-holder-Lin}
\|T^n x-T^n y\|_{\mathrm L}\leq\kappa_n\|x-y\|_{\mathrm L}^{\alpha_n}
\end{equation}
for every \(x,y\in K\) and every \(n\geq1\)?
\end{problem}

Lin's result does not directly answer Problem~\ref{prob:lin-holder}. For each fixed \(\alpha<1\), the quantity 
\(d^\alpha\) is larger than \(d\) when \(0<d<1\), so that every H\"older estimate allows expansion at small scales. When 
\(\alpha_n\to1\), however, this additional freedom disappears at every fixed positive scale. The following proposition makes this observation precise.

\subsection{A necessary macroscopic condition}

\begin{proposition}[Asymptotic nonexpansiveness away from the diagonal]
\label{prop:macroscopic}
Let \(K\) be bounded, and suppose that \(T\colon K\to K\) satisfies \eqref{eq:variable-holder-Lin}. Then, for every \(r>0\),
\[
\limsup_{n\to\infty}\sup_{\substack{x,y\in K\\ \|x-y\|_{\mathrm L}\geq r}}\frac{\|T^n x-T^n y\|_{\mathrm L}}{\|x-y\|_{\mathrm L}}\leq1.
\]
Thus \(T\) is asymptotically nonexpansive outside every fixed neighborhood of the diagonal.
\end{proposition}

\begin{proof}
If \(\|x-y\|_{\mathrm L}\geq r\), then since \(\alpha_n-1<0\), we have
\[
\frac{\|T^n x-T^n y\|_{\mathrm L}}{\|x-y\|_{\mathrm L}}\leq\kappa_n\|x-y\|_{\mathrm L}^{\alpha_n-1}\leq \kappa_n r^{\alpha_n-1}.
\]
It follows that
\[
\sup_{\substack{x,y\in K\\ \|x-y\|_{\mathrm L}\geq r}}\frac{\|T^n x-T^n y\|_{\mathrm L}}{\|x-y\|_{\mathrm L}}\leq\kappa_n r^{\alpha_n-1}\longrightarrow1.
\]
\end{proof}

In particular, any affirmative solution to Problem~\ref{prob:lin-holder} must be asymptotically nonexpansive at
every fixed positive scale. Any expansion that persists asymptotically must therefore be confined to a neighborhood of the diagonal whose radius tends to zero.

\begin{remark}
Although Proposition~\ref{prop:macroscopic} is formulated in terms of Lin's norm, its proof is entirely metric. More precisely, the same conclusion holds in any metric space after replacing $\|\cdot\|_{\mathrm L}$ by the corresponding metric. In particular,
no linear or geometric property specific to Lin's renorming is used.
\end{remark}

\subsection{Obstructions to shift-type constructions}

Let
\[
K=\left\{x=(x_j)\in\ell_1:x_j\geq 0,\quad\sum_{j=1}^{\infty}x_j=1\right\}.
\]
By \eqref{eq:Lin-equivalence}, the set \(K\) is closed, bounded, and convex in \((\ell_1,\|\cdot\|_{\mathrm L})\). Define the shift-type maps \(S,W\colon K \to K\) by
\[
S(x_1,x_2,x_3,\ldots)=(0,x_1,x_2,x_3,\ldots)
\]
and
\[
W(x):=\Bigg(1-\sum_{j=1}^{\infty}w_jx_j\Bigg)e_1+\sum_{j=1}^{\infty}w_jx_je_{j+1},\qquad w_j:=\frac{\gamma_j}{\gamma_{j+1}}.
\]
Both mappings have no fixed points. However, neither of them can solve Problem~\ref{prob:lin-holder}.

\begin{proposition}
\label{prop:shift-obstructions}
The following assertions hold.
\begin{enumerate}
    \item The canonical right shift $S$ is not asymptotically
    H\"older nonexpansive.

    \item The balanced weighted shift $W$ is affine and fixed point
    free, and satisfies
    \[
     \sup_{n\geq1}\operatorname{Lip}_{\mathrm L}(W^n)\leq\frac1{\gamma_1}.
    \]
    Nevertheless, $W$ is not asymptotically H\"older nonexpansive.
\end{enumerate}
\end{proposition}

\begin{proof}
For the first assertion, fix $1\leq i<j$. From the definition of
Lin's norm,
\[
\|e_i-e_j\|_{\mathrm L}=\max\{2\gamma_i,\gamma_j\}=:d_{i,j}<2.
\]
On the other hand,
\[
S^ne_i=e_{i+n},\qquad S^ne_j=e_{j+n},
\]
and hence
\[
\|S^ne_i-S^ne_j\|_{\mathrm L}=\max\{2\gamma_{i+n},\gamma_{j+n}\}\longrightarrow2.
\]
If $S$ satisfied
\[
\|S^nx-S^ny\|_{\mathrm L}\leq \kappa_n\|x-y\|_{\mathrm L}^{\alpha_n},\qquad\kappa_n\to1,\quad\alpha_n\to 1,
\]
then, by applying this estimate to $e_i$ and $e_j$ and passing to the limit, we would obtain $2\leq d_{i,j}$, contradicting $d_{i,j}<2$.

\vskip .1cm 

We now consider $W$. Since $0<w_j\leq1$, the mapping $W$ preserves
nonnegativity and total mass, so $W(K)\subset K$. It is also affine.
If $Wx=x$, then
\[
x_{j+1}=w_jx_j
\]
for every $j\geq1$. Consequently,
\[
x_j=x_1\prod_{r=1}^{j-1}w_r=x_1\frac{\gamma_1}{\gamma_j}.
\]
Since $\gamma_j\to1$, this sequence cannot belong to $\ell_1$ unless
$x_1=0$. In the latter case, $x=0\notin K$. Thus $W$ is fixed point
free. Let $P\colon\ell_1\to\ell_1$ be the associated linear operator,
defined by
\[
Ph=\left(\sum_{j=1}^{\infty}(1-w_j)h_j\right)e_1+\sum_{j=1}^{\infty}w_jh_je_{j+1}.
\]
The operator $P$ is positive and preserves the $\ell_1$-mass of
positive vectors. Therefore, $\|P^nh\|_1\leq\|h\|_1$ for every $h\in\ell_1$ and every $n\geq1$. Since
\[
\gamma_1\|h\|_1 \leq\|h\|_{\mathrm L} \leq\|h\|_1,
\]
we obtain
\[
\|W^nx-W^ny\|_{\mathrm L}=\|P^n(x-y)\|_{\mathrm L}\leq\frac1{\gamma_1}\|x-y\|_{\mathrm L}.
\]
Thus
\[
\sup_{n\geq1}\operatorname{Lip}_{\mathrm L}(W^n)\leq\frac1{\gamma_1}.
\]
It remains to prove that $W$ is not asymptotically H\"older
nonexpansive. Choose $j>1$ such that $\gamma_j>\gamma_1$, and set
\[
d_j:=\|e_1-e_j\|_{\mathrm L}=\max\{2\gamma_1,\gamma_j\}.
\]
The mass starting at the $i$-th coordinate and undergoing no return to the first coordinate during the first $n$ iterations is
\[
p_{i,n}:=\prod_{r=i}^{i+n-1}w_r=\frac{\gamma_i}{\gamma_{i+n}}.
\]
All trajectories undergoing at least one return are supported in the
first $n$ coordinates after $n$ iterations. Hence, by considering the
tail beginning at the $(n+1)$-st coordinate,
\begin{align*}
    \|W^ne_1-W^ne_j\|_{\mathrm L}
    &\geq
    \gamma_{n+1}\bigl(p_{1,n}+p_{j,n}\bigr)\\
    &=
    \gamma_1+
    \frac{\gamma_{n+1}\gamma_j}{\gamma_{n+j}}.
\end{align*}
It follows that
\[
    \liminf_{n\to\infty}
    \|W^ne_1-W^ne_j\|_{\mathrm L}
    \geq
    \gamma_1+\gamma_j.
\]
Since $\gamma_j>\gamma_1$,
\[
\gamma_1+\gamma_j>\max\{2\gamma_1,\gamma_j\}=d_j.
\]
Thus $W$ exhibits persistent expansion at the fixed positive scale $d_j$, contradicting the necessary condition in
Proposition~\ref{prop:macroscopic}. Therefore $W$ is not asymptotically H\"older nonexpansive.
\end{proof}

\subsection{A sufficient construction strategy}
\label{subsec:lin-sufficient-strategy}

Proposition~\ref{prop:macroscopic} gives a necessary condition for an affirmative solution of Problem~\ref{prob:lin-holder}: the iterates must become asymptotically nonexpansive at every fixed positive scale. Theorem~\ref{thm:off-diagonal-holder}, proved in the preceding section, shows that a suitably uniform version of this condition is also sufficient when combined with uniform Lipschitz control. We thus do not restate that theorem here, but record its consequences for Lin's renorming. More precisely, a sufficient route to solving Problem~\ref{prob:lin-holder} is to construct a nonempty, closed, bounded, and convex set
\[
K\subset(\ell_1,\|\cdot\|_{\mathrm L}),\qquad\operatorname{diam}_{\mathrm L}(K)\leq1,
\]
and a fixed-point-free mapping \(T\colon K\to K\) satisfying the following two conditions:
\begin{equation}\label{eq:Lin-uniform-Lipschitz-route}
\sup_{n\geq1}\operatorname{Lip}_{\mathrm L}(T^n)<\infty,
\end{equation}
and there exist null positive sequences \((\delta_n)\) and \((\eta_n)\) such that
\begin{equation}\label{eq:Lin-off-diagonal-route}
\|T^nx-T^ny\|_{\mathrm L}\leq (1+\eta_n)\|x-y\|_{\mathrm L}
\end{equation}
whenever
\[
\|x-y\|_{\mathrm L}\geq\delta_n.
\]
Indeed, Theorem~\ref{thm:off-diagonal-holder} would then provide
sequences
\[
    \alpha_n\in(0,1),
    \qquad
    \kappa_n>1,
    \qquad
    \alpha_n\to1,
    \qquad
    \kappa_n\to1,
\]
such that
\[
    \|T^nx-T^ny\|_{\mathrm L}
    \leq
    \kappa_n\|x-y\|_{\mathrm L}^{\alpha_n}
\]
for every \(x,y\in K\) and every \(n\geq1\). Since \(T\) is fixed point
free, this would give an affirmative answer to Problem~\ref{prob:lin-holder}. The metric content of this strategy can be expressed through the
off-diagonal Lipschitz profile
\[
    \Lambda_n(r)
    :=
    \sup_{\substack{x,y\in K\\
                    \|x-y\|_{\mathrm L}\geq r}}
    \frac{\|T^nx-T^ny\|_{\mathrm L}}
         {\|x-y\|_{\mathrm L}},
    \qquad r>0.
\]
Condition \eqref{eq:Lin-off-diagonal-route} amounts to finding a
sequence \(\delta_n\to0\) such that
\[
    \Lambda_n(\delta_n)\leq1+\eta_n
    \qquad\text{with}\qquad
    \eta_n\to0.
\]
Thus the desired mapping need not be asymptotically nonexpansive
uniformly over all distinct pairs. Equivalently, it is enough to find a sequence $\delta_n\to0$ such
that
\[
\limsup_{n\to\infty}\Lambda_n(\delta_n)\leq1,
\]
or, equivalently,
\[
\bigl(\Lambda_n(\delta_n)-1\bigr)_+\longrightarrow0.
\]
Thus the macroscopic distortion need not converge exactly to one; it may remain strictly below one. The bound
\eqref{eq:Lin-uniform-Lipschitz-route} controls those pairs near the diagonal, and the H\"older exponent supplied by
Theorem~\ref{thm:off-diagonal-holder} absorbs the resulting fixed Lipschitz factor. This reformulation also clarifies the obstructions established in Proposition~\ref{prop:shift-obstructions}. Both the canonical right shift \(S\) and the balanced
weighted shift \(W\) admit pairs at a fixed positive distance whose images undergo persistent asymptotic expansion. Consequently, for neither mapping can one find a sequence \(\delta_n\to0\) such that the corresponding off-diagonal profiles satisfy
\[
\limsup_{n\to\infty}\Lambda_n(\delta_n)\leq 1.
\]
A successful construction must therefore be genuinely scale dependent: the region in which expansion larger than \(1+\eta_n\) occurs must move toward the diagonal as \(n\to\infty\). Accordingly, Theorem~\ref{thm:off-diagonal-holder} isolates the remaining construction problem but does not solve it. One must still produce, in Lin's renormed \(\ell_1\), a single fixed-point-free
uniformly Lipschitzian mapping whose macroscopic distortion improves along its iterates. The periodic construction discussed in the next subsection provides useful uniformly Lipschitzian dynamics, but its exact periodicity prevents the required asymptotic improvement.

\subsection{Comparison with the Acuaviva--Acuaviva construction}

Acuaviva and Acuaviva \cite{AA} construct, on a broad class of decomposable Banach spaces, a fixed-point-free mapping \(U\) on the unit ball such that
\[
\inf_x\|x-Ux\|=0\qquad\text{and}\qquad\sup_{n\geq1}\operatorname{Lip}(U^n)<\infty.
\]
A distinctive feature of their construction is the algebraic
periodicity $U^2=S$, $U^3=U$, which implies
\[
    U^{2m}=S,
    \qquad
    U^{2m+1}=U
    \qquad(m\geq1).
\]
This periodicity is effective for controlling all iterates by a common Lipschitz constant, but it becomes an obstruction when both
\(\alpha_n\) and \(\kappa_n\) are required to converge to \(1\).

\begin{proposition}
\label{prop:periodic-obstruction}
Let \(K\) have the fixed point property for nonexpansive mappings, and let \(T\colon K\to K\). Suppose that there is an unbounded sequence \((n_j)\) such that
\[
T^{n_j}=T.
\]
If
\[
\|T^n x-T^n y\|\leq\kappa_n\|x-y\|^{\alpha_n},\qquad\alpha_n\to1,\qquad\kappa_n\to1,
\]
then \(T\) has a fixed point.
\end{proposition}

\begin{proof}
For every \(x,y\in K\),
\[
\|Tx-Ty\|=\|T^{n_j}x-T^{n_j}y\|\leq\kappa_{n_j}\|x-y\|^{\alpha_{n_j}}.
\]
Letting \(j\to\infty\), we obtain $\|Tx-Ty\|\leq\|x-y\|$. Thus \(T\) is nonexpansive and consequently has a fixed point.
\end{proof}

\subsection{Conclusion and remaining problem}

The results of this section delimit the principal obstruction to a fixed-point-free construction in Lin's renormed $\ell_1$. The gauge criterion of Theorem~\ref{thm:l1-gauge-criterion} relies on the exact splitting property of the canonical $\ell_1$-norm and therefore does not transfer directly to $\|\cdot\|_{\mathrm L}$. On the other hand, Proposition~\ref{prop:shift-obstructions} rules out the shift-type maps, while Proposition~\ref{prop:periodic-obstruction} shows that an exactly periodic construction of the Acuaviva--Acuaviva type cannot satisfy variable-exponent estimates with both $\alpha_n\to1$ and $\kappa_n\to1$. Theorem~\ref{thm:off-diagonal-holder} nevertheless provides a concrete sufficient mechanism. It would be enough to construct a
fixed-point-free uniformly Lipschitzian mapping whose iterates become asymptotically nonexpansive outside a neighborhood of the diagonal whose radius tends to zero. Thus the remaining problem is not merely to modify the standard shift, but to produce dynamics whose macroscopic distortion genuinely improves along the iterates.

It remains open whether there exist a nonempty closed bounded convex
set
\[
K\subset(\ell_1,\|\cdot\|_{\mathrm L})
\]
and a fixed-point-free mapping $T\colon K\to K$ satisfying
\[
\|T^nx-T^ny\|_{\mathrm L}\leq\kappa_n\|x-y\|_{\mathrm L}^{\alpha_n},\qquad\alpha_n\to1,\quad\kappa_n\to1.
\]


\section{Examples}\label{sec:5}

This section provides examples of mappings satisfying Definition \ref{def:1.1}. Our goal is to illustrate the versatility of this class and to show that it encompasses operators failing classical regularity conditions, such as global continuity or Lipschitzianity.

\smallskip 

\begin{example}[\cite{DBen1}]\label{ex:3.1}
Let $T : [0, 1] \to [0, 1]$ be defined by $T(x) = 0$ if $x < 1$ and $T(1) = 1/2$. Then $T$ is clearly a discontinuous mapping which satisfies $T^n(x) = 0$ for all $n \ge 2$. Since \(T^n=0\) for every \(n\geq2\), the mapping \(T\) is of
asymptotically H\"older nonexpansive type. Notice, however, that \(T\) is not asymptotically H\"older nonexpansive in the sense of
Definition~\ref{def:1.1}(i).
\end{example}


\begin{example}[A non-Lipschitzian mapping which is asymptotically H\"older nonexpansive]\label{ex:3.2}
Let \(0<r\leq1/2\), and consider the positive part of the closed ball of radius \(r\) in \(\ell_2\):
\[
K:=B_{\ell_2^+}(r)=\left\{x=(t_j)\in\ell_2:t_j\geq0,\quad \|x\|_2\leq r\right\}.
\]
Let \((\alpha_n)\subset(0,1)\) be nondecreasing and satisfy \(\alpha_n\to1\). Choose a sequence \((B_n)_{n\geq2}\subset(0,1)\)
such that
\begin{equation}\label{eq:product-B-zero}
\prod_{n=2}^{\infty}B_n=0.
\end{equation}
For \(n\geq2\), put
\[
A_n:=\alpha_{n-1}B_n\qquad\text{and}\qquad p_n:=\frac{\alpha_n}{\alpha_{n-1}}.
\]
Since \((\alpha_n)\) is nondecreasing, \(p_n\geq1\). Define \(T\colon K\to\ell_2\) by
\begin{equation}\label{eq:example-32-map}
T(x)=(r-\|x\|_2)t_1^{\alpha_1}e_2+\sum_{n=2}^{\infty}A_n t_n^{p_n}e_{n+1},\qquad x=(t_n)\in K.
\end{equation}
We first verify that \(T(K)\subset K\). Since \(0\leq t_n\leq r\leq1\), we have
\[
0\leq A_nt_n^{p_n}\leq t_n\qquad(n\geq2).
\]
Moreover,
\[
0\leq (r-\|x\|_2)t_1^{\alpha_1}\leq r-\|x\|_2.
\]
Consequently,
\begin{align*}
\|T(x)\|_2^2&=(r-\|x\|_2)^2t_1^{2\alpha_1}+\sum_{n=2}^{\infty}A_n^2t_n^{2p_n}\\
&\leq (r-\|x\|_2)^2+\sum_{n=2}^{\infty}t_n^2\leq (r-\|x\|_2)^2+\|x\|_2^2.
\end{align*}
Therefore \(\|T(x)\|_2\leq r\), and all the coordinates of \(T(x)\) are nonnegative. Thus \(T(x)\in K\). Notice also that 
\(T(0)=0\). Nevertheless, \(T\) is not Lipschitz. Indeed, for \(0<t<r\), let \(x_t=te_1\). Then
\[
T(x_t)=(r-t)t^{\alpha_1}e_2,\qquad T(0)=0,
\]
and hence
\[
\frac{\|T(x_t)-T(0)\|_2}{\|x_t\|_2}=(r-t)t^{\alpha_1-1}\longrightarrow\infty\qquad (t\downarrow0).
\]
We next estimate the iterates of \(T\). For \(n\geq2\), define
\[
\varphi_n(t):=A_nt^{p_n},\qquad 0\leq t\leq r.
\]
Since
\[
\varphi_n'(t)=A_np_nt^{p_n-1}=\alpha_nB_n t^{p_n-1},
\]
and \(0\leq t\leq r\leq1\), we have
\[
0\leq\varphi_n'(t)\leq B_n.
\]
It follows that
\begin{equation}\label{eq:example-32-scalar-Lipschitz}
|\varphi_n(t)-\varphi_n(s)|\leq B_n|t-s|\qquad(s,t\in[0,r]).
\end{equation}
A direct induction gives, for \(m\geq2\),
\begin{align}
T^m(x)&=\left(\prod_{i=2}^{m}A_i^{\alpha_m/\alpha_i}\right)(r-\|x\|_2)^{\alpha_m/\alpha_1}t_1^{\alpha_m}e_{m+1}\notag\\
&\hskip 1cm +\sum_{k=2}^{\infty}\left(\prod_{i=k}^{k+m-1}A_i^{\alpha_{k+m-1}/\alpha_i}\right)t_k^{\alpha_{k+m-1}/\alpha_{k-1}}e_{k+m}.
\label{eq:example-32-iterates}
\end{align}
Let \(x=(t_k)\) and \(y=(s_k)\) belong to \(K\), and put $d:=\|x-y\|_2$. Since \(\operatorname{diam}(K)\leq2r\leq1\), we have \(0\leq d\leq1\). We first estimate the term in \eqref{eq:example-32-iterates} arising from the first coordinate. Set
\[
q_m:=\frac{\alpha_m}{\alpha_1}.
\]
For $a=r-\|x\|_2$ and $b=r-\|y\|_2$, we have \(a,b\in[0,1]\), and $|a-b|\leq d$. Therefore,
\begin{align*}
\left|a^{q_m}t_1^{\alpha_m}-b^{q_m}s_1^{\alpha_m}\right|
&\leq |a^{q_m}-b^{q_m}|t_1^{\alpha_m}+b^{q_m}|t_1^{\alpha_m}-s_1^{\alpha_m}|\\[1mm]
&\leq q_m|a-b|+|t_1-s_1|^{\alpha_m}\\[1mm]
&\leq\frac1{\alpha_1}d+d^{\alpha_m}.
\end{align*}
Since \(0\leq d\leq1\) and \(\alpha_m<1\), we have \(d\leq d^{\alpha_m}\). Hence
\begin{equation*}\label{eq:example-32-first-term}
\left|a^{q_m}t_1^{\alpha_m}-b^{q_m}s_1^{\alpha_m}\right|
\leq \left(\frac1{\alpha_1}+1\right)d^{\alpha_m}.
\end{equation*}
Moreover, since \(A_i\leq B_i<1\) and \(\alpha_m/\alpha_i\geq1\) for \(2\leq i\leq m\),
\[
\prod_{i=2}^{m}A_i^{\alpha_m/\alpha_i}\leq\prod_{i=2}^{m}B_i.
\]
Thus the difference between the \((m+1)\)-st coordinates of \(T^m(x)\) and \(T^m(y)\) is bounded by
\begin{equation}\label{eq:example-32-first-coordinate-final}
\left(\frac1{\alpha_1}+1\right)\left(\prod_{i=2}^{m}B_i\right) d^{\alpha_m}.
\end{equation}
For \(k\geq2\), the contribution of the \(k\)-th coordinate to the \(m\)-th iterate is obtained by composing $\varphi_k,\varphi_{k+1},\ldots,\varphi_{k+m-1}$. It follows from \eqref{eq:example-32-scalar-Lipschitz} that
\begin{equation}\label{eq:example-32-tail-coordinate}
\left|(\varphi_{k+m-1}\circ\cdots\circ\varphi_k)(t_k)-(\varphi_{k+m-1}\circ\cdots\circ\varphi_k)(s_k)\right|\leq\left(\prod_{i=k}^{k+m-1}B_i\right)|t_k-s_k|.
\end{equation}
In particular, because \(0<B_i<1\),
\[
\left|(\varphi_{k+m-1}\circ\cdots\circ\varphi_k)(t_k)-(\varphi_{k+m-1}\circ\cdots\circ\varphi_k)(s_k)\right|\leq|t_k-s_k|.
\]
Combining \eqref{eq:example-32-first-coordinate-final} and \eqref{eq:example-32-tail-coordinate}, we obtain
\begin{align*}
\|T^m(x)-T^m(y)\|_2^2
&\leq\left(\frac1{\alpha_1}+1\right)^2\left(\prod_{i=2}^{m}B_i^2\right)d^{2\alpha_m}+\sum_{k=2}^{\infty}\left(\prod_{i=k}^{k+m-1}B_i^2\right)|t_k-s_k|^2\\
&\leq\left(\frac1{\alpha_1}+1\right)^2\left(\prod_{i=2}^{m}B_i^2\right)d^{2\alpha_m}+d^2.
\end{align*}
Since \(d\leq d^{\alpha_m}\), it follows that
\begin{align*}
\|T^m(x)-T^m(y)\|_2&\leq\left[1+\left(\frac1{\alpha_1}+1\right)\prod_{i=2}^{m}B_i\right]d^{\alpha_m}.
\end{align*}
By \eqref{eq:product-B-zero},
\[
\kappa_m:=1+\left(\frac1{\alpha_1}+1\right)\prod_{i=2}^{m}B_i\longrightarrow 1.
\]
Therefore,
\[
\|T^m(x)-T^m(y)\|_2\leq \kappa_m\|x-y\|_2^{\alpha_m},
\]
and \(T\) is asymptotically H\"older nonexpansive, although it is not
Lipschitz.
\end{example}

\smallskip 

The approach in \cite{DBen1} shows that the class of asymptotically H\"older nonexpansive type maps strictly contains those that are asymptotically H\"older nonexpansive. 

\begin{example}[{\cite[Example 3.3]{DBen1}}]\label{ex:3.3} Let $(e_i)$ and $\ell_2$ be as above, set $u_n=e_n/2^n$, and consider the compact convex set $K=\overline{\mathrm{co}}\{u_n\colon n\in\mathbb{N}\}$. Then every $x=(x_k)\in K$ has the representation $x=\sum_{n=1}^\infty \lambda_k u_k$, where $0\leq \lambda_n\leq 1$, $\sum_{n=1}^\infty \lambda_n\leq 1$ and $x(n)=\lambda_n/2^n$. Denote $P=\{n\in\mathbb{N}\colon n=2^j \text{ for } j\in\mathbb{N}\}$ and $Q=\mathbb{N}\setminus P$. Also, write $\mu_k=0$ if $\lambda_k<1$ and $\mu_k=1/2$ if $\lambda_k=1$. The mapping $T\colon K\to K$ defined by
\[
T\Bigg( \sum_{n=1}^\infty \lambda_n u_n\Bigg)= \sum_{n\in Q} \lambda_n u_{n+1} + \sum_{n\in P} \mu_n u_{n+1},
\]
has none of its positive iterates continuous and is of asymptotically H\"older nonexpansive type. 
\end{example}

Next, we provide a non-Lipschitzian mapping that is uniformly asymptotically H\"older $\kappa$-Lipschitz, but which does not satisfy condition (i) of Definition \ref{def:1.1}.

\begin{example}\label{ex:3.4}
Let $K:=\big\{x=(t_n)\in c_0:0\leq t_n\leq1\ \text{for every }n\big\}$. Let \((\alpha_n)\subset(0,1)\) be strictly increasing and satisfy \(\alpha_n\to1\). Define \(T\colon K\to c_0\) by
\begin{equation}\label{eq:example-34-map}
T(x)=e_1+t_1^{\alpha_2}e_2+\sum_{i=2}^{\infty}t_i^{\alpha_{i+1}/\alpha_i}e_{i+1},\qquad x=(t_i)\in K.
\end{equation}
We first verify that \(T(K)\subset K\). Since
\[
\frac{\alpha_{i+1}}{\alpha_i}>1
\]
and \(0\leq t_i\leq1\), we have
\[
0\leq t_i^{\alpha_{i+1}/\alpha_i}\leq t_i.
\]
As \(t_i\to0\), it follows that $t_i^{\alpha_{i+1}/\alpha_i}\to 0$. Thus \(T(x)\in c_0\), and all its coordinates belong to \([0,1]\). Hence \(T(x)\in K\). The mapping \(T\) is fixed point free. Indeed, if \(T(x)=x\), then  $t_1=1$, $t_2=t_1^{\alpha_2}=1$, and inductively
\[
t_{i+1}=t_i^{\alpha_{i+1}/\alpha_i}=1\qquad(i\geq2).
\]
This contradicts \(x\in c_0\). The mapping is also non-Lipschitz. For \(0<t<1\), let \(x_t=te_1\). Then
\[
\|T(x_t)-T(0)\|_\infty=t^{\alpha_2},
\]
and hence
\[
\frac{\|T(x_t)-T(0)\|_\infty}{\|x_t-0\|_\infty}=t^{\alpha_2-1}\longrightarrow\infty
\]
as \(t\downarrow0\). A straightforward induction gives
\begin{equation}\label{eq:example-34-iterates}
T^n(x)=\sum_{j=1}^{n}e_j+t_1^{\alpha_{n+1}}e_{n+1}+\sum_{i=2}^{\infty}t_i^{\alpha_{n+i}/\alpha_i}e_{n+i}.
\end{equation}
Let \(x=(t_i)\) and \(y=(s_i)\) belong to \(K\), and put $d:=\|x-y\|_\infty$. For the term arising from the first coordinate, the elementary inequality
\[
|a^\theta-b^\theta|\leq|a-b|^\theta,\qquad a,b\in[0,1],\quad 0<\theta\leq1,
\]
gives
\begin{equation}\label{eq:example-34-first-coordinate}
|t_1^{\alpha_{n+1}}-s_1^{\alpha_{n+1}}|\leq |t_1-s_1|^{\alpha_{n+1}}\leq d^{\alpha_{n+1}}.
\end{equation}
For \(i\geq2\), set
\[
q_{n,i}:=\frac{\alpha_{n+i}}{\alpha_i}>1.
\]
By the mean value theorem,
\begin{align*}
|t_i^{q_{n,i}}-s_i^{q_{n,i}}|&\leq q_{n,i}|t_i-s_i|\leq\frac1{\alpha_i}|t_i-s_i|.
\end{align*}
Since \(0\leq|t_i-s_i|\leq d\leq1\),
\[
|t_i-s_i|\leq |t_i-s_i|^{\alpha_{n+1}}\leq d^{\alpha_{n+1}}.
\]
Moreover, \(\alpha_i\geq\alpha_2\) for \(i\geq2\). Therefore,
\begin{equation}\label{eq:example-34-tail-coordinate}
|t_i^{q_{n,i}}-s_i^{q_{n,i}}|\leq \frac1{\alpha_2}d^{\alpha_{n+1}}.
\end{equation}
It follows from \eqref{eq:example-34-iterates}, \eqref{eq:example-34-first-coordinate}, and \eqref{eq:example-34-tail-coordinate} that
\begin{equation}\label{eq:example-34-uniform-holder}
\|T^n(x)-T^n(y)\|_\infty\leq\frac1{\alpha_2} \|x-y\|_\infty^{\alpha_{n+1}}.
\end{equation}
Since \(\alpha_{n+1}\to1\), \(T\) is uniformly asymptotically H\"older with constant \(1/\alpha_2\).

\smallskip 

We finally show that \(T\) is not asymptotically H\"older nonexpansive. Fix \(h\in(0,1)\), and consider the vectors $u=e_1+e_2$, $v=e_1+(1-h)e_2$. Both \(u\) and \(v\) belong to \(K\), and
\[
\|u-v\|_\infty=h.
\]
The coordinates of \(u\) and \(v\) differ only at the second coordinate. Hence \eqref{eq:example-34-iterates} gives
\[
\|T^n(u)-T^n(v)\|_\infty=1-(1-h)^{\alpha_{n+2}/\alpha_2}.
\]
Since \(\alpha_{n+2}\to1\),
\begin{equation}\label{eq:example-34-expanded-limit}
\|T^n(u)-T^n(v)\|_\infty\longrightarrow 1-(1-h)^{1/\alpha_2}.
\end{equation}
Because \(1/\alpha_2>1\) and \(0<1-h<1\),
\[
(1-h)^{1/\alpha_2}<1-h,
\]
and therefore
\begin{equation}\label{eq:example-34-strict-expansion}
1-(1-h)^{1/\alpha_2}>h.
\end{equation}
Suppose, contrary to our assertion, that there existed sequences
\[
(\beta_n)\subset(0,1),\qquad(\varepsilon_n)\subset[0,\infty),\qquad\beta_n\to1,\qquad\varepsilon_n\to0,
\]
such that
\[
\|T^n x-T^n y\|_\infty\leq(1+\varepsilon_n)\|x-y\|_\infty^{\beta_n}\qquad(x,y\in K).
\]
Applying this inequality to \(u\) and \(v\) and passing to the limit would give
\[
1-(1-h)^{1/\alpha_2}\leq\lim_{n\to\infty}(1+\varepsilon_n)h^{\beta_n}=h,
\]
contradicting \eqref{eq:example-34-strict-expansion}. Thus \(T\) is not asymptotically H\"older nonexpansive. 

\smallskip 
Consequently, this example separates the class of uniformly asymptotically H\"older mappings from the smaller class of asymptotically H\"older nonexpansive mappings.
\end{example}

\begin{example}[A genuinely non-Lipschitz mapping that is asymptotically
nonexpansive away from the diagonal]
\label{ex:away-from-diagonal}
Let \(K=[0,1]\) and define \(T\colon K\to K\) by 
\[
T(0)=0,\qquad T(x)=\frac{x}{2}+\frac{x}{4}\sin\left(\frac{1}{x^{2}}\right),\qquad 0<x\leq1.
\]
Then \(T\) is continuous and
\[
\frac{x}{4}\leq T(x)\leq\frac{3x}{4},\qquad x\in[0,1].
\]
In particular, \(T(K)\subset K\), and induction gives
\begin{equation}\label{eq:orbit-decay-example}
0\leq T^n(x)\leq\left(\frac34\right)^n x\leq\left(\frac34\right)^n\qquad(x\in K).
\end{equation}
We claim that \(T\) is asymptotically nonexpansive away from the diagonal. Indeed, for every \(r>0\), if \(|x-y|\geq r\), then
\[
|T^n(x)-T^n(y)|\leq T^n(x)+T^n(y)\leq 2\left(\frac34\right)^n.
\]
Consequently,
\[
\sup_{\substack{x,y\in K\\ |x-y|\geq r}}\frac{|T^n(x)-T^n(y)|}{|x-y|}\leq\frac{2}{r}\left(\frac34\right)^n\longrightarrow0.
\]
Thus, more strongly,
\[
\lim_{n\to\infty}\sup_{\substack{x,y\in K\\ |x-y|\geq r}}\frac{|T^n(x)-T^n(y)|}{|x-y|}=0\qquad\text{for every }r>0.
\]
The example is nontrivial in the sense that \(T\) is not Lipschitz. For \(x>0\),
\[
T'(x)=\frac12+\frac14\sin\left(\frac{1}{x^2}\right)-\frac{1}{2x^2}\cos\left(\frac{1}{x^2}\right).
\]
Taking
\[
x_k=\frac{1}{\sqrt{2\pi k}},
\]
we obtain
\[
T'(x_k)=\frac12-\pi k,
\]
and hence
\[
\sup_{0<x\leq1}|T'(x)|=\infty.
\]
Therefore \(T\) is not even Lipschitz at the first iterate, despite being asymptotically contractive on every set of pairs staying a positive distance away from the diagonal.
\end{example}

\begin{remark}
This example also shows that asymptotic nonexpansiveness away from the
diagonal, by itself, does not provide the uniform Lipschitz control
required in Theorem~\ref{thm:off-diagonal-holder}.
\end{remark}

\medskip

\paragraph{Acknowledgement} The authors would like to thank their colleague Luciano G. Josino for pointing out some inaccuracies in an earlier version of the manuscript.

\medskip 

\subsection*{AI statement} The overall content of this manuscript remains largely unchanged from the first version. For this revision, the authors employed OpenAI's ChatGPT 5.6 Sol for editorial refinement, organization, and consistency  checks. Specifically, the tool was employed to clarify and expand upon the arguments in \cite{Karlovitz}, as well as to address the technical issues underlying Theorem 3.7 and Section \ref{sec:4}. All AI-assisted outputs were independently evaluated and verified by the authors, who retain full responsibility for the contents of this paper.


\end{document}